\documentclass[12pt]{article}
\usepackage{amssymb}
\usepackage{amsmath}
\usepackage{amsfonts}

\newtheorem{theorem}{Theorem}

\newtheorem{lemma}[theorem]{Lemma}

\newtheorem{proposition}[theorem]{Proposition}

\newenvironment{proof}[1][Proof]{\noindent\textbf{#1.} }{\ \rule{0.5em}{0.5em}}
\begin{document}

\title{Variations and extensions of the Gaussian concentration inequality,
Part I}
\author{Daniel J. Fresen\thanks{%
University of Pretoria, Department of Mathematics and Applied Mathematics,
daniel.fresen@up.ac.za. 2010 \textit{MSC}:
60E05, 60E15, 52A27, 52A30. \textit{Keywords}: Gaussian concentration inequality,
heavy tailed random variables.}}
\date{}
\maketitle

\begin{abstract}
The classical Gaussian concentration inequality for Lipschitz functions is
adapted to a setting where the
classical assumptions (i.e. Lipschitz and Gaussian) are not met. The theory
is more direct than much of the existing theory designed to handle related
generalizations. An application is presented to linear combinations
of heavy tailed random variables.
\end{abstract}

\section{Introduction}

Recall the Gaussian concentration inequality in one of its most classical
forms: if $\psi:\mathbb{R}^{n}\rightarrow \mathbb{R}$ is a Lipschitz function, and $Z$ is
a random vector in $\mathbb{R}^{n}$ with the standard normal distribution,
then for all $t>0$,%
\begin{equation}
\mathbb{P}\left\{ \left\vert \psi\left( Z\right) -\mathbb{M}\psi\left( Z\right)
\right\vert >t\right\} \leq C\exp \left( \frac{-ct^{2}}{Lip\left( \psi\right)
^{2}}\right)  \label{first eq Gaussian}
\end{equation}%
where $C,c>0$ are universal constants and $Lip\left( \psi\right)$ is the Lipschitz constant of $\psi$. $\mathbb{M}\psi\left( Z\right) $ denotes
the median of $\psi\left( Z\right) $, and can be replaced with the mean $%
\mathbb{E}\psi\left( Z\right) $. It follows from the Gaussian isoperimetric
inequality of \cite{SuTs} and \cite{Bor75} that
this can be improved to%
\begin{equation*}
\mathbb{P}\left\{ \psi\left( Z\right) -\mathbb{M}\psi\left( Z\right) >t\right\}
\leq 1-\Phi \left( \frac{t}{Lip\left( \psi\right) }\right)
\end{equation*}%
where $\Phi $ is the standard normal cumulative distribution. Equality
clearly holds when $\psi$ is linear.

Assuming for simplicity that $\psi$ is $C^{1}$%
, it follows from a result of Pisier \cite[Theorem 2.2 p176]{Pis0} that if $%
Y $ is another random vector in $\mathbb{R}^{n}$ with the standard normal
distribution, independent of $Z$, then for any convex function $\varphi :%
\mathbb{R}\rightarrow \mathbb{R}$,%
\begin{equation*}
\mathbb{E}\varphi \left( \psi\left( Z\right) -\psi\left( Y\right) \right) \leq 
\mathbb{E}\varphi \left( \frac{\pi }{2}\left\vert \nabla \psi\left( Z\right)
\right\vert Z'\right)
\end{equation*}%
where $Z'$ has the standard normal distribution in $\mathbb{R}$ and is
independent of $Z$. This often implies that one can compare the tails of $%
\psi\left( Z\right) -\psi\left( Y\right) $ to that of $\frac{\pi }{2}\left\vert
\nabla \psi\left( Z\right) \right\vert Z'$, which in turn leads to a bound on
the tails of $\psi\left( Z\right) -\mathbb{M}\psi\left( Z\right) $.

Pisier's
version of Gaussian concentration has three advantages over the classical
version. The first is that the proof (with input from Maurey) is
quite simple. The second is that the Lipschitz condition has been
removed. The third is that the resulting concentration inequality is determined by the distribution of $\left\vert \nabla \psi\left( Z\right)
\right\vert $ rather than $Lip(\psi)=\sup\left\vert \nabla \psi\left( \cdot\right)
\right\vert$. This is really the way things should be, considering that the Lipschitz constant of a function might be determined by behaviour on sets of small measure that have little effect on concentration properties of $\psi(Z)$. This can be seen in the examples
\[
u(x)=\left\vert x \right\vert+\max\left\{0,\varepsilon-\frac{\left\vert x \right\vert}{\varepsilon}\right\} \hspace{2.5cm}v(x)=\left\vert x \right\vert+\frac{\varepsilon \sqrt{\left\vert x \right\vert}}{1+\sqrt{\left\vert x \right\vert}}
\]
for $x\in\mathbb{R}^n$ and $\varepsilon >0$.

In two papers, 'Variations and extensions of the Gaussian concentration
inequality' Part I (here) and Part II (to appear elsewhere), we study
ways in which to apply the above inequalities in settings where they have not previously been applied because they appear to be inapplicable. The result is that the Gaussian concentration
inequality gives rise to various inequalities that seem to have
nothing to do with the normal distribution. This paper, Part I, focuses more
on the classical version, while Part II focuses on Pisier's version. The
theory is from scratch, in that no background is needed, other than\ a basic
knowledge of mathematics (the one thing that we don't prove is
the Gaussian concentration inequality itself). Our opinion is that it is
more direct than much of the modern theory of concentration of measure,
including for example the theory of Poincar\'{e} and log-Sobolev
inequalities. We refer the reader to \cite{BLMass, Ledx} for information on the
concentration of measure phenomenon.

In Section \ref{methaa} we present the main ideas of the paper. In Section %
\ref{heavy tail}, as an illustration of these ideas, we prove a concentration inequality for
linear combinations of heavy tailed random variables:
if $0<q\leq 1$ and $\left( X_{i}\right)_1^n$ is a sequence of independent symmetric random variables with
$\mathbb{P}\left\{ \left\vert X_{i}\right\vert >t\right\} = \exp \left(
-t^{q}\right) $ for all $t>0$, then
\begin{equation}
\mathbb{P}\left\{ \left\vert \sum_{i=1}^{n}a_iX_{i}\right\vert >t\right\} \leq
2\exp \left( -\min \left\{C^{-1/q}q^{2/q} \left( \frac{t}{\left\vert a\right\vert }%
\right) ^{2},C^{-1}\left( \frac{t}{\left\vert a\right\vert _{\infty }}\right)
^{q}\right\} \right)\label{dev est q}
\end{equation}%
(to say that $X_i$ is symmetric means that $X_i$ and $-X_i$ have the same distribution). This ties in naturally with known results in the case $1\leq
q<\infty $ since the dual of $\ell _{q}^{n}$ is isometric to $\ell _{\infty
}^{n}$ when $0<q\leq 1$. Considering that
\begin{equation*}
C_1^{1/q}q^{-2/q}\left\vert a\right\vert^2 \leq \mathrm{Var}\left( \sum_{i=1}^{n}a_iX_{i} \right)\leq C_2^{1/q}q^{-2/q}\left\vert a\right\vert^2
\end{equation*}%
we see that the coefficient $C^{-1/q}q^{2/q}$ is sharp up to the value of $C$. By applying symmetrization and contraction techniques it can be applied
in the non-symmetric case as long as $\mathbb{E}X_i=0$ and $\mathbb{P}\left\{ \left\vert X_{i}\right\vert >t\right\} \leq C\exp \left(
-t^{q}\right)$, with $C^{-1/q}q^{2/q}$ replaced by a different quantity. This follows work of Hitczenko, Montgomery-Smith and Oleszkiewicz \cite{HitMont1997} who estimated the $p^{th}$ moment of $\left\vert\sum a_i X_i\right\vert$. The authors of \cite{HitMont1997} presented a significantly weaker result:
\[
\lim_{t\rightarrow\infty }\log _{t}\ln \left[\mathbb{P}\left\{ \left\vert
\sum_{i=1}^{n}a_iX_{i}\right\vert >t\right\}\right]^{-1} =q.
\]
See Cor. 6.5 in their paper. The case of equal coefficients, which avoids the main obstacle, is considered in \cite{ALPT2012}, see their Eq. (3.6). A similar inequality appears in \cite{KlSo2012} (see their Sec. 4) but does not include the sub-Gaussian part. The classical case $q\geq 1$ is covered in \cite{BLMass} (their Ex. 2.27 p.50) but we did not see any mention of the case $q\in(0,1)$. Of course in the vast array of techniques available there are undoubtedly several ways to prove (\ref{dev est q}), we present two, and the result may be folklore.  More applications, mostly non-linear, are given in \cite{Fre, Fre Lor} and will appear in future papers.

\subsection*{Comments and notation}

$\mathbb{M}$ denotes median, and $C$, $c$, $%
C^{\prime }$ etc. usually denote positive
universal constants that may represent different values at each appearance.
Dependence on variables will usually be indicated by subscripts, $C_{q}$, $%
c_{q}$ etc. The Euclidean norm of a vector $a\in \mathbb{R}^n$ is denoted $\left\vert a\right\vert$ while $\left\vert a \right\vert _{p}=\left( \sum_1^n \left\vert a_i \right\vert^p \right)^{1/p} $ is its $\ell _p^n$-norm $(1\leq p <\infty)$, and $\left\vert a \right\vert _{\infty}=\mathrm{max}\{\left\vert a_i\right\vert \} $. Set-builder notation such as $\left\{x\in\mathbb{R}^n:\psi(x)\leq R\right\}$ will at times be simplified to $\left\{\psi\leq R\right\}$. The term 'random variable' will be used exclusively for real
valued random variables. Let $\gamma _{n}$ denote the standard
Gaussian measure on $\mathbb{R}^{n}$ with density
\[
\frac{d\gamma _{n}}{dx}=(2\pi
)^{-n/2}\exp \left( -\left\vert x\right\vert ^{2}/2\right).
\]
Let $\Phi $ denote the standard normal cumulative distribution and $\phi
=\Phi ^{\prime }$ the standard normal density. The expression $Z\sim N(0,I_n)$ means that the random vector $Z$ follows the standard normal distribution on $\mathbb{R}^n$ (with mean $0$ and identity covariance matrix $I_n$). The 'Gaussian' in 'Gaussian concentration' refers to $Z$, not $\psi\left( Z\right) $,
and includes the case where $\psi\left( Z\right) $ has thicker-than-Gaussian tails, which may occur when $\psi$ grows more
rapidly than linear. The $C^1$ condition that appears from time to time is stronger than necessary; as long as $f$ is locally Lipschitz it will be differentiable almost everywhere, and we will often disregard certain points of non-differentiability without explicitly commenting on the matter. The sigma algebra on $\mathbb{R}^n$ is always assumed to be the Borel sigma algebra, so \textit{measurable} means \textit{Borel measurable}. The following lemma will be used implicitly.

\begin{lemma}
\label{lil alg}If $f,g:\left[ 0,\infty \right) \rightarrow \left[ 0,\infty
\right) $ are continuous strictly increasing functions with $f\left(
0\right) =g\left( 0\right) =0$, $t\in \left[ 0,\infty \right) $ and $s=\max
\left\{ f\left( t\right) ,g\left( t\right) \right\} $ then $t=\min \left\{
f^{-1}\left( s\right) ,g^{-1}\left( s\right) \right\} $.
\end{lemma}

\section{\label{methaa}The methodology}

\subsection{\label{intro Gauss conc}Gaussian concentration for non-Lipschitz functions}

In this Section we discuss an alternative technique for handling the case when $\psi$ is non-Lipschitz, or when the Lipschitz constant of $\psi$ is much larger than the average local Lipschitz constant, so (\ref{first eq Gaussian}) either does not apply or becomes crude. The technique is outlined as follows: modify $\psi$ on a set of small probability on which it behaves badly so that the modified function $\psi^*$ is Lipschitz, apply classical Gaussian concentration to $\psi^*$, and then transfer the result for $\psi^*$ back to $\psi$. We now describe this in more detail.

Consider the setting where $Z$ has the
standard normal distribution in $\mathbb{R}^{n}$ and $\psi:\mathbb{R}%
^{n}\rightarrow \mathbb{R}$ is, say, $C^{1}$, but is not assumed to be
Lipschitz. Restrict $\psi$ to some set $E$ such that $Z\in E$ with high
probability and $\left\vert \nabla \psi\right\vert $ is bounded nicely on $E$.
The set $E=\left\{ x\in \mathbb{R}^{n}:\left\vert \nabla \psi\left( x\right)
\right\vert \leq R\right\} $ is often a natural choice (with appropriate $R>0$). Assuming that $E$ is convex (we discuss the non-convex case next), we have $Lip\left( \psi|_{E}\right) \leq  \sup_{E}\left\vert
\nabla \psi\right\vert $. One can then extend the
restriction $\psi|_{E}$ to the entire space $\mathbb{R}^{n}$ so that the
extension $\psi^{\ast }$ obeys $Lip\left( \psi^{\ast }\right) =Lip\left( \psi|_{E}\right) \leq 
R$. By
applying classical Gaussian concentration of Lipschitz functions as in (\ref{first eq Gaussian}) to $\psi^{\ast
}$ and observing that $\mathbb{P}\left\{ \psi\left( Z\right) =\psi^{\ast }\left(
Z\right) \right\} \geq \mathbb{P}\left\{ Z\in E\right\} $, we may transfer
the concentration inequality for $\psi^{\ast }\left( Z\right) $ about $\mathbb{M%
}\psi^{\ast }\left( Z\right) $ to an inequality for $\psi\left( Z\right) $ about $%
\mathbb{M}\psi\left( Z\right) $.

The idea of proving concentration inequalities by modifying a function on a
set where it behaves badly is not new. The technique based on restriction involving $\left\vert \nabla
\psi\right\vert $ and Lipschitz extension is found in \cite[Lemma 2.2]{ASY}, \cite[Section 3]{Adam15}, \cite[Section 6]{BNT} and \cite[Section 5]{MeSz}, and is related to an observation contained in unpublished
lecture notes that we prepared and distributed at Yale in 2012 and 2014. In these papers one assumes that $E$ is convex, or at least that $Lip\left( \psi|_{E}\right) \leq  \sup_{E}\left\vert
\nabla \psi\right\vert $ (which is not necessarily the case when $E$ is not convex). In \cite[Corollary%
]{Gra} and \cite[Section 3]{Vu} a related procedure is also used, but in a slightly different setting where the measure is supported on $\left[
0,1\right] ^{n}$ and the Lipschitz constant is taken with respect to the
Hamming distance, see the bottom of p. 264 there.

\subsection{\label{non-convex}When $E=\{\left\vert\nabla \psi\right\vert\leq R\}$ is not convex}

Without convexity one cannot apply the bound $Lip(\psi|_E)\leq \sup_E \left\vert\nabla \psi\right\vert$. This is seen, for example, by considering a set in the shape of the letter C, where the two ends of the C are close together so that the C looks like an O, and a function which increases gradually as we move from one end of the C to the other. We consider a parameter $\frak{L}(E)$, which measures the distance one has to travel inside $E$ to get from a point $x$ to a point $y$ as compared to the Euclidean distance, this ratio maximized over all pairs of distinct points in $E$. In Proposition \ref{lip grad} we show that $Lip(\psi|_E)\leq \frak{L}(E) \sup_E \left\vert\nabla \psi\right\vert$ and that $\frak{L}(E)$ is the correct coefficient. But this is not enough: we must provide useful bounds on $\frak{L}(E)$, which we do in Proposition \ref{distorted boxy}. We show that whenever $E$ is a certain non-affine deformation of an unconditional convex body (as the inverse image under a continuous coordinate-wise transformation), then $\frak{L}(E)\leq \sqrt{2}$.

Another approach is to find a convex subset $K\subseteq E$ and then apply the bound $Lip(\psi|_K)\leq \sup_K \left\vert\nabla \psi\right\vert\leq \sup_E \left\vert\nabla \psi\right\vert$. Ideally one would want the Gaussian measure of $K$ to be comparable to the Gaussian measure of $E$, in the sense that
\[
\gamma_n\left(\mathbb{R}^n\setminus K\right)\leq C\gamma_n\left(\mathbb{R}^n\setminus C^{-1}E\right)
\]
for a universal constant $C>0$. In Section \ref{mainy resty} we present some very basic observations in this direction, motivated by the methodology in Sections \ref{intro Gauss conc} and \ref{intro non Gauss conc} and of independent interest. We refer the reader to \cite{Tal1, Tal2} for a deeper discussion of the topic of approximating star bodies by convex sets, including various open problems.

\subsection{\label{intro non Gauss conc}Non-Gaussian
concentration}

If $X$ is a random vector in $\mathbb{R}^{n}$ with any distribution $\mu $,
then $X$ has the same distribution as $T(Z)$ for some measurable function $T:\mathbb{R}%
^{n}\rightarrow \mathbb{R}^n$, where $Z$ is a random vector with the standard
normal distribution on $\mathbb{R}^{n}$, denoted $\gamma_n$. One such map is the Kn%
\"{o}the-Rosenblatt rearrangement, see for example \cite{Vil}, and in some cases one can
write down an explicit formula for $T$. For our purposes we may assume that $X=T(Z)$, so $f(X)=(f\circ T)(Z)$. Under fairly general conditions we may then apply Gaussian concentration
to $\psi=f\circ T$ to obtain a concentration inequality for $f(X)$. Let us reiterate this in words: \textit{any function of any random vector has the same distribution as a function of a Gaussian random vector}.

Transportation methods, and a variety of them, are quite standard in the theory of concentration of measure, see for example \cite{BLMass, CatGoGuRo, Gozl, Ledx, Tal3, Tangu}. The above procedure in particular is alluded to in \cite[p. 1046]{Na} in the context where $\mu$ is a Lipschitz image
of the standard Gaussian measure, in anticipation of applying the classical Gaussian concentration inequality for Lipschitz functions. Note that if one can apply Gaussian
concentration to a wider class of functions $\psi$ as in Section \ref{intro
Gauss conc} above, then one may apply Gaussian concentration to a wider class
of non-Gaussian measures. The observations in Section \ref{intro Gauss conc} and
those of this section therefore work particularly well \textit{together} and
each increases the usefulness of the other. This synergy is at
the heart of the paper.

One is thus left with the problem of finding a good choice of $T$. If $\mu $
is an $n$-fold product measure, then the most natural $T$ acts
coordinate-wise in the obvious manner (and is the Kn\"{o}the-Rosenblatt
rearrangement). If $\mu $ is spherically symmetric, then the most natural $T$
acts radially. For most other measures, we expect that the Kn\"{o}%
the-Rosenblatt rearrangement is not a good choice. In the case of
log-concave measures, the Brenier map (see for example \cite{KlKo}) may be
better, although the Brenier map is best for minimizing a transportation
cost, which is not exactly what we want.

\section{\label{heavy tail}Application: Concentration of linear combinations
of heavy tailed random variables}

Concentration of linear combinations of independent random variables is most
classically studied under the assumption of exponential integrability, i.e. $%
\mathbb{E}\exp \left( \varepsilon X_{i}\right) <\infty $ for some $%
\varepsilon >0$. In this context, the exponential moment method plays an
essential role: Using Markov's inequality and independence,%
\begin{eqnarray*}
\mathbb{P}\left\{ \sum_{i=1}^{n}a_{i}X_{i}>t\right\} &=&\mathbb{P}\left\{
\exp \left( \lambda \sum_{i=1}^{n}a_{i}X_{i}\right) >\exp \left( \lambda
t\right) \right\} \leq \exp \left( -\lambda t\right) \mathbb{E}\exp \left(
\lambda \sum_{i=1}^{n}a_{i}X_{i}\right) \\
&=&\exp \left( -\lambda t\right) \prod_{i=1}^{n}\mathbb{E}\exp \left(
\lambda a_{i}X_{i}\right).
\end{eqnarray*}%
The resulting estimate is then optimized over $\lambda >0$ such that $%
\mathbb{E}\exp \left( \lambda a_{i}X_{i}\right) <\infty $ for all $i$.
Outside the realm of exponential integrability (still assuming
independence and linearity), one would estimate power moments and use Markov's inequality,%
\begin{equation}
\mathbb{P}\left\{ \left\vert \sum_{i=1}^{n}a_{i}X_{i}\right\vert >t\right\} =%
\mathbb{P}\left\{ \left\vert \sum_{i=1}^{n}a_{i}X_{i}\right\vert
^{p}>t^{p}\right\} \leq t^{-p}\mathbb{E}\left\vert
\sum_{i=1}^{n}a_{i}X_{i}\right\vert ^{p}.  \label{power moment method}
\end{equation}%
Power moments of sums are more difficult to compute than exponential moments because power functions lack the critical property of the exponential function as used in the exponential moment method. If we assume that each $X_{i}$ has a symmetric distribution and that $%
\mathbb{P}\left\{ \left\vert X_{i}\right\vert \geq t\right\} =\exp \left(
-N_{i}(t)\right) $ for some concave function $N:\left[ 0,\infty \right)
\rightarrow \left[ 0,\infty \right) $, then it was shown in \cite[Theorem 1.1]{HitMont1997} that for
all $p\geq 2$,%
\begin{equation}
\left( \mathbb{E}\left\vert \sum_{i=1}^{n}a_{i}X_{i}\right\vert ^{p}\right)
^{1/p}\leq C\left( \left( \sum_{i=1}^{n}\left\vert a_{i}\right\vert ^{p}%
\mathbb{E}\left\vert X_{i}\right\vert ^{p}\right) ^{1/p}+\sqrt{p}\left(
\sum_{i=1}^{n}\left\vert a_{i}\right\vert ^{2}\mathbb{E}\left\vert
X_{i}\right\vert ^{2}\right) ^{1/2}\right)  \label{HMSO}
\end{equation}%
where $C>0$ is a universal constant, with a corresponding lower bound with $%
C $ replaced by a different constant $c>0$. General upper and lower bounds for the $p^{th}$ moment of a sum of independent random variables were later presented in \cite{Lat}, reducing the problem of computing these moments to that of evaluating a certain Orlicz norm type expression.

In the special case where $%
\left( X_{i}\right) _{1}^{n}$ are i.i.d. symmetric Weibull variables with $\mathbb{P}\left\{ \left\vert X_{i}\right\vert \geq t\right\} =\exp \left(
-t^{q}\right) $, for $0<q\leq 1$, we
show below that in (\ref{HMSO}), $\left\vert a\right\vert _{p}$ can be replaced with $%
\left\vert a\right\vert _{\infty }$ at the cost of $C^{1/q}$, i.e.
\begin{equation}
\left( \mathbb{E}\left\vert
\sum_{i=1}^{n}a_{i}X_{i}\right\vert ^{p}\right) ^{1/p}\leq \left(\frac{C}{q}\right)^{1/q}\left(
p^{1/q}\left\vert a\right\vert _{\infty }+\sqrt{p}\left\vert a\right\vert\right)
\label{replace with infty norm}
\end{equation}%
which by Markov's inequality leads to the tail estimate 
\begin{equation}
\mathbb{P}\left\{ \left\vert \sum_{i=1}^{n}a_iX_{i}\right\vert >t\right\} \leq
2\exp \left( -c_{q}\min \left\{ \left( \frac{t}{\left\vert a\right\vert }%
\right) ^{2},\left( \frac{t}{\left\vert a\right\vert _{\infty }}\right)
^{q}\right\} \right)  \label{Weibul dist for 0<q<=1}
\end{equation}%
where $c_{q}>0$ is made explicit in Theorem \ref{Weibul conc theory}. This is more subtle than observing that the $\ell_p^n$ norm is equivalent to the $\ell_\infty^n$ norm for $p>c\ln n$ and comes down to the geometry of Minkowski sums of the form $uB_{\ell_1}+vB_{\ell_2}$ and the question of when they contain $B_{\ell_r}$ for any given $1<r<2$ ($n$ does not play a role here).

If $X_{1}$ is any random variable and $X_{1}^{\prime }$ is an independent
copy of $X_{1}$, and if $a>0$ is such that $\mathbb{P}\left\{ \left\vert
X_{1}\right\vert >a\right\} \leq 1/2$, then%
\begin{equation*}
\left\{ X_{1}>t+a\right\} \cap \left\{ X_{1}^{\prime }\leq a\right\}
\subseteq \left\{ X_{1}-X_{1}^{\prime }>t\right\} \subseteq \left\{
X_{1}>t/2\right\} \cup \left\{ X_{1}^{\prime }<-t/2\right\}
\end{equation*}%
so by independence and identical distributions,%
\begin{equation}
\frac{1}{2}\mathbb{P}\left\{ X_{1}>t+a\right\} \leq \mathbb{P}\left\{
X_{1}-X_{1}^{\prime }>t\right\} \leq \mathbb{P}\left\{ \left\vert
X_{1}\right\vert >t/2\right\}.  \label{symm prince}
\end{equation}%
The significance is that $X_{1}-X_{1}^{\prime }$ is symmetric. This can be
combined with the following contraction principle, see \cite[Lemma 4.6]%
{LeTal} for a more general version:

\begin{lemma}
\label{contraction principle}Let $\varphi :\left[ 0,\infty \right)
\rightarrow \left[ 0,\infty \right) $ be a convex function, $K_{1}\geq 1$, $%
K_{2}>0$, and let $\left( X_{i}\right) _{1}^{n}$ and $\left( Y_{i}\right)
_{1}^{n}$ each be i.i.d. sequences of symmetric random variables with $%
\mathbb{P}\left\{ \left\vert X_{i}\right\vert >t\right\} \leq K_{1}\mathbb{P}%
\left\{ K_{2}\left\vert Y_{i}\right\vert >t\right\} $ for all $i$ and all $%
t>0$. Then for all $a\in \mathbb{R}^{n}$,%
\begin{equation*}
\mathbb{E}\varphi \left( \left\vert \sum_{i=1}^{n}a_{i}X_{i}\right\vert
\right) \leq \mathbb{E}\varphi \left( K_{1}K_{2}\left\vert
\sum_{i=1}^{n}a_{i}Y_{i}\right\vert \right).
\end{equation*}
\end{lemma}

If we are given an i.i.d. sequence of random variables $\left( X_{i}\right)
_{1}^{n}$ that satisfy a tail bound such as $\mathbb{P}\left\{ \left\vert
X_{i}\right\vert >t\right\} \leq h(t)$, then we consider the symmetrized
sequence $\left( X_{i}-X_{i}^{\prime }\right) _{1}^{n}$ which obeys a
similar tail bound, apply known results for a specific sequence of random
variables $\left( Y_{i}\right) _{1}^{n}$ with similar tails (e.g. Weibull
variables), compare $\mathbb{E}\varphi \left( \left\vert
\sum_{i=1}^{n}a_{i}\left( X_{i}-X_{i}^{\prime }\right) \right\vert \right) $
to $\mathbb{E}\varphi \left( K_{1}K_{2}\left\vert
\sum_{i=1}^{n}a_{i}Y_{i}\right\vert \right) $ using Lemma \ref{contraction
principle}, convert this to a bound on $\mathbb{P}\left\{ \left\vert
\sum_{i=1}^{n}a_{i}\left( X_{i}-X_{i}^{\prime }\right) \right\vert
>t\right\} $ (using say Markov's inequality), and then transfer the result
for $\sum_{i=1}^{n}a_{i}X_{i}-\sum_{i=1}^{n}a_{i}X_{i}^{\prime }$ back to a
bound on $\mathbb{P}\left\{ \left\vert \sum_{i=1}^{n}a_{i}X_{i}\right\vert
>t\right\} $ using (\ref{symm prince}). In this way, estimates such as (\ref%
{Weibul dist for 0<q<=1}) may be extended to the case of tail bounds such as 
$\mathbb{P}\left\{ \left\vert X_{i}\right\vert >t\right\} \leq C\exp \left(
-t^{q}\right) $.

\begin{proof}[Proof of (\protect\ref{replace with infty norm}) and (\protect
\ref{Weibul dist for 0<q<=1})]
By H\"{o}lder's inequality and optimization, for all $p>2$ and $s\in(0,\infty)$,
\begin{eqnarray*}
\left\vert a\right\vert _{p}\leq \left\vert a\right\vert _{\infty
}^{1-2/p}\left\vert a\right\vert ^{2/p}\leq \omega(s)\left\vert a \right\vert_\infty+s\left\vert a \right\vert
\end{eqnarray*}
where
\[
\omega(s)=\left\{
\begin{array}{lll}
(2/p)^{2/(p-2)}(1-2/p)s^{-2/(p-2)} &:& 0<s\leq2/p\\
1-s &:& 2/p<s\leq 1\\
0 &:& s>1
\end{array}.
\right.
\]
To see this, assume momentarily and without loss of generaility that $\left\vert a\right\vert=1$ and optimize $(r^{1-2/p}-s)/r$ over $r\in[0,1]$. Setting $s=2p^{1/2-1/q}$,
\[
\left\vert a\right\vert _{p}\leq\left\{
\begin{array}{lll}
(1-2/p)p^{-(3q-2)/[(p-2)q]}\left\vert a \right\vert_\infty+2p^{1/2-1/q}\left\vert a \right\vert &:& 0< q<2/3\\
\left\vert a \right\vert_\infty+2p^{1/2-1/q}\left\vert a \right\vert &:& 2/3\leq q\leq 1
\end{array}.
\right.
\]
We use this unless $0< q<2/3$ and $2\leq p\leq 3$ (3 is arbitrary), in which case we use $\left\vert a\right\vert _{p}\leq \left\vert a\right\vert$. The result is that $\left\vert a\right\vert_p\leq C^{1/q}\left(\left\vert a\right\vert_\infty+p^{1/2-1/q}\left\vert a\right\vert \right)$, and this clearly holds for $p=2$ as well. (\ref{replace with infty norm}) now follows from (\ref{HMSO}) using $\mathbb{E}\left\vert X_1\right\vert^r=(r/q)\Gamma(r/q)$ and $cs\leq \Gamma(s)^{1/s}\leq Cs$ for all $s\geq 2$. The
probability bound follows by optimizing over $p$. In Case 1 we assume that $%
p^{1/q}\left\vert a\right\vert _{\infty }\leq \sqrt{p}$, and then%
\begin{equation*}
\mathbb{P}\left\{ \left\vert \sum_{i=1}^{n}a_iX_{i}\right\vert >t\right\} \leq
t^{-p}\mathbb{E}\left\vert \sum_{i=1}^{n}a_{i}X_{i}\right\vert ^{p}\leq
\left( \frac{C_{q}\sqrt{p}}{t}\right) ^{p}=\exp \left( -c_{q}t^{2}\right)
\end{equation*}%
for $p=c_{q}t^{2}$. This value of $p$ satisfies the defining inequality of
Case 1 if $t\leq c_{q}\left\vert a\right\vert _{\infty }^{-q/(2-q)}$. In
Case 2 we assume that $p^{1/q}\left\vert a\right\vert _{\infty }\geq \sqrt{p}
$, and then%
\begin{equation*}
\mathbb{P}\left\{ \left\vert \sum_{i=1}^{n}a_iX_{i}\right\vert >t\right\} \leq
t^{-p}\mathbb{E}\left\vert \sum_{i=1}^{n}a_{i}X_{i}\right\vert ^{p}\leq
\left( \frac{C_{q}p^{1/q}\left\vert a\right\vert _{\infty }}{t}\right)
^{p}\leq \exp \left( -C_{q}\frac{t^{q}}{\left\vert a\right\vert _{\infty
}^{q}}\right)
\end{equation*}%
for $p=\left( C_{q}^{-1}\left\vert a\right\vert _{\infty }^{-1}t\right)
^{q}e^{-1}$. This value of $p$ is allowed in Case 2 provided $t\geq
C_{q}\left\vert a\right\vert _{\infty }^{-q/(2-q)}$. For $c_{q}\left\vert
a\right\vert _{\infty }^{-q/(2-q)}\leq t\leq C_{q}\left\vert a\right\vert
_{\infty }^{-q/(2-q)}$, the result follows by adjusting the values of $c_{q}$
and $C_{q}$ and using the fact that the cumulative distribution is
non-decreasing.
\end{proof}

We now present a direct proof of (\ref{Weibul dist for 0<q<=1}) without
using (\ref{replace with infty norm}) or the results of \cite{HitMont1997}. The methodology used is that outlined in Section \ref{methaa} and does not make use of linearity. The result is dimension independent and applies in the infinite dimensional setting with $a\in\ell_2$ (the partial sums are Cauchy in $L^2$).

\begin{theorem}\label{Weibul conc theory}
There exists $C>0$ such that the following is true. Let $n\in \mathbb{N}$, $%
0<q\leq 1$, $a\in \mathbb{R}^{n}$, $a\neq 0$, and $\left( X_{i}\right)
_{1}^{n}$ an i.i.d. sequence of symmetric Weibull random variables with
parameter $q$, i.e. $\mathbb{P}\left\{ \left\vert X_{i}\right\vert
>t\right\} =\exp \left( -t^{q}\right) $, $t\geq 0$. Then for all $t>0$,%
\begin{equation}
\mathbb{P}\left\{ \left\vert \sum_{i=1}^{n}a_iX_{i}\right\vert >t\right\} \leq
C\exp \left( -\min \left\{ C^{-1/q}q^{2/q}\left( \frac{t}{\left\vert
a\right\vert }\right) ^{2},C^{-1}\left( \frac{t}{\left\vert a\right\vert
_{\infty }}\right) ^{q}\right\} \right).\label{devdev}
\end{equation}
\end{theorem}

\begin{proof}
Write $X=\left( F^{-1}\Phi \left( Z_{i}\right) \right) _{i=1}^{n}$, where $%
F(t)=\mathbb{P}\left\{ X_{1}\leq t\right\} $ and $Z$ is a random vector in $%
\mathbb{R}^{n}$ with the standard normal distribution, and define $\psi
\left( x\right) =\sum_{1}^{n}a_{i}F^{-1}\Phi \left( x_{i}\right) $, so that $%
\sum_{1}^{n}a_{i}X_{i}=\psi \left( Z\right) $, and%
\begin{equation*}
\left\vert \nabla \psi \left( x\right) \right\vert =\left(
\sum_{i=1}^{n}\left( \frac{a_{i}\phi \left( \Phi ^{-1}\left( \Phi
x_{i}\right) \right) }{f\left( F^{-1}\left( \Phi x_{i}\right) \right) }%
\right) ^{2}\right) ^{1/2}
\end{equation*}%
where $f=F^{\prime }$. Now, by comparing derivatives and behaviour at
infinity,%
\begin{equation*}
\frac{\phi (t)}{1+t}\leq 1-\Phi \left( t\right) \leq \frac{\phi (t)}{t}:t>0
\end{equation*}%
which implies that for $1/2\leq t<1$, $\phi \left( \Phi ^{-1}\left( t\right)
\right) \leq \left( 1+\Phi ^{-1}\left( t\right) \right) \left( 1-t\right) $.
Now $1-\Phi \left( t\right) \leq \phi (t)$ for $t\geq 1$, so for $1/2\leq
t<1 $ we have (taking $\max =1$ when second argument undefined), 
\begin{equation*}
1+\Phi ^{-1}\left( t\right) \leq 1+\max \left\{ 1,\sqrt{2\ln \frac{1}{\sqrt{%
2\pi }\left( 1-t\right) }}\right\}.
\end{equation*}%
By direct computation, for the same range of $t$,%
\begin{equation*}
f\left( F^{-1}\left( t\right) \right) =q\left( 1-t\right) \left( \ln \frac{1%
}{2\left( 1-t\right) }\right) ^{-\left( -1+1/q\right) }
\end{equation*}%
so%
\begin{equation*}
\left( \frac{\phi \left( \Phi ^{-1}\left( t\right) \right) }{f\left(
F^{-1}\left( t\right) \right) }\right) ^{2}\leq q^{-2}\left( A\Phi
^{-1}\left( t\right) +B\right) ^{-2+4/q}
\end{equation*}%
for universal constants $A,B>0$, and%
\begin{eqnarray}
\left\vert \nabla \psi \left( x\right) \right\vert &=&\left(
\sum_{i=1}^{n}\left( \frac{a_{i}\phi \left( \Phi ^{-1}\left( \Phi \left\vert
x_{i}\right\vert \right) \right) }{f\left( F^{-1}\left( \Phi \left\vert
x_{i}\right\vert \right) \right) }\right) ^{2}\right) ^{1/2}\nonumber\\
&\leq &q^{-1}\left( \sum_{i=1}^{n}a_{i}^{2}\left( A\left\vert
x_{i}\right\vert +B\right) ^{-2+4/q}\right) ^{1/2}. \label{grad bound in spec case}
\end{eqnarray}%
The function%
\begin{equation*}
\left] x\right[ =\left( \sum_{i=1}^{n}a_{i}^{2}\left\vert x_{i}\right\vert
^{-2+4/q}\right) ^{q/(4-2q)}
\end{equation*}%
is a norm, so%
\begin{eqnarray*}
Lip\left( \left] \cdot \right[ \right) &=&\sup \left\{ \left] \theta \right[
:\theta \in S^{n-1}\right\} \\
&=&\sup \left\{ \left( \sum_{i=1}^{n}a_{i}^{2}\left\vert x_{i}\right\vert
^{-1+2/q}\right) ^{q/(4-2q)}:x_{i}\geq 0,\sum_{i=1}^{n}x_{i}=1\right\} \\
&=&\left\vert a\right\vert _{\infty }^{q/(2-q)}.
\end{eqnarray*}%
The Lipschitz constant of $x\mapsto \left] \left( A\left\vert
x_{i}\right\vert +B\right) _{1}^{n}\right[ $ is therefore at most $%
A\left\vert a\right\vert _{\infty }^{q/(2-q)}$, and by classical Gaussian
concentration applied to this function, with probability at least $1-C\exp
\left( -c\lambda ^{2}\right) $,%
\begin{eqnarray*}
&&\left( \sum_{i=1}^{n}a_{i}^{2}\left( A\left\vert Z_{i}\right\vert
+B\right) ^{-2+4/q}\right) ^{q/(4-2q)} \\
&\leq &\mathbb{E}\left( \sum_{i=1}^{n}a_{i}^{2}\left( A\left\vert
Z_{i}\right\vert +B\right) ^{-2+4/q}\right) ^{q/(4-2q)}+A\left\vert
a\right\vert _{\infty }^{q/(2-q)}\lambda \\
&\leq &\left( \sum_{i=1}^{n}a_{i}^{2}\mathbb{E}\left( A\left\vert
Z_{1}\right\vert +B\right) ^{-2+4/q}\right) ^{q/(4-2q)}+A\left\vert
a\right\vert _{\infty }^{q/(2-q)}\lambda \\
&=&\left( \mathbb{E}\left( A\left\vert Z_{1}\right\vert +B\right)
^{-2+4/q}\right) ^{q/(4-2q)}\left\vert a\right\vert ^{q/(2-q)}+A\left\vert
a\right\vert _{\infty }^{q/(2-q)}\lambda
\end{eqnarray*}%
i.e.%
\begin{eqnarray}
&&\left( \sum_{i=1}^{n}a_{i}^{2}\left( A\left\vert Z_{i}\right\vert
+B\right) ^{-2+4/q}\right) ^{1/2}\nonumber \\
&\leq &\left[ \left( \mathbb{E}\left( A\left\vert Z_{1}\right\vert +B\right)
^{-2+4/q}\right) ^{q/(4-2q)}\left\vert a\right\vert ^{q/(2-q)}+A\left\vert
a\right\vert _{\infty }^{q/(2-q)}\lambda \right] ^{-1+2/q}.\label{noname123}
\end{eqnarray}%
Define
\begin{eqnarray*}
K&=&\left\{ x\in \mathbb{R}^{n}:q^{-1}\left( \sum_{i=1}^{n}a_{i}^{2}\left( A\left\vert
x_{i}\right\vert +B\right) ^{-2+4/q}\right) ^{1/2} \leq R\right\}
\end{eqnarray*}
where%
\begin{eqnarray}
R &=&q^{-1}\left[ \left( \mathbb{E}\left( A\left\vert Z_{1}\right\vert
+B\right) ^{-2+4/q}\right) ^{q/(4-2q)}\left\vert a\right\vert
^{q/(2-q)}+A\left\vert a\right\vert _{\infty }^{q/(2-q)}\lambda \right]
^{-1+2/q}\nonumber \\
&\leq &q^{-1}2^{-1+2/q}\left( \mathbb{E}\left( A\left\vert Z_{1}\right\vert
+B\right) ^{-2+4/q}\right) ^{1/2}\left\vert a\right\vert
+q^{-1}(2A)^{-1+2/q}\left\vert a\right\vert _{\infty }\lambda ^{-1+2/q}\nonumber \\
&\leq &C_{2}^{1/q}\left( 1/q\right) ^{1/q}\left\vert a\right\vert
+q^{-1}(2A)^{-1+2/q}\left\vert a\right\vert _{\infty }\lambda ^{-1+2/q}.\label{R lambda}
\end{eqnarray}%
By (\ref{noname123}) and its probability bound, $\gamma _{n}\left( K\right) \geq 1-C\exp \left( -c\lambda
^{2}\right)$. Since $-2+4/q\geq 2$ the function $x\mapsto \sum_{i=1}^{n}a_{i}^{2}\left( A\left\vert
x_{i}\right\vert +B\right) ^{-2+4/q}$ is the sum of convex functions and is therefore itself convex, which implies that $K$ is convex. Since $K$ is convex and, by (\ref{grad bound in spec case}), $\left\vert \nabla \psi \left( x\right)
\right\vert \leq R$ on $K$, the restricted function $\psi |_{K}$ is Lipschitz, with corresponding Lipschitz constant at most $R$.
Let $\psi ^{\sharp }$ denote a Lipschitz extension of the
restriction $\psi |_{K}$ such that $Lip\left( \psi ^{\sharp }\right) \leq R$%
. The existence of extensions of real valued Lipschitz
functions preserving the Lipschitz constant is a basic result in the theory of metric spaces, see e.g. \cite{Matt}. Applying Gaussian concentration to $\psi ^{\sharp }$,%
\begin{equation*}
\mathbb{P}\left\{ \left\vert \psi ^{\sharp }\left( Z\right) -\mathbb{M}\psi
^{\sharp }\left( Z\right) \right\vert >\lambda R\right\} \leq C\exp \left(
-c\lambda ^{2}\right)
\end{equation*}%
yet $\mathbb{P}\left\{ \psi ^{\sharp }\left( Z\right) \neq \psi \left(
Z\right) \right\} \leq C\exp \left( -c\lambda ^{2}\right) $ so $\mathbb{P}%
\left\{ \left\vert \psi \left( Z\right) -\mathbb{M}\psi ^{\sharp }\left(
Z\right) \right\vert >\lambda R\right\} \leq 2C\exp \left( -c\lambda
^{2}\right) $. Assuming that this probability bound is less than $1/2$,
\begin{eqnarray*}
\mathbb{P}%
\left\{  \psi \left( Z\right) <\mathbb{M}\psi ^{\sharp }\left(
Z\right) -\lambda R\right\}<1/2\\
\mathbb{P}%
\left\{  \psi \left( Z\right) >\mathbb{M}\psi ^{\sharp }\left(
Z\right) +\lambda R\right\}<1/2
\end{eqnarray*}
which implies that $\left\vert \mathbb{M}\psi \left( Z\right) -\mathbb{M}\psi ^{\sharp }\left(
Z\right) \right\vert\leq\lambda R$. The bound for concentration of $\psi\left(Z\right)$ around $\mathbb{M}\psi ^{\sharp }\left(
Z\right)$ can now be written as a concentration inequality around $ \mathbb{M}\psi \left( Z\right)$, i.e.
\[
\mathbb{P}%
\left\{ \left\vert \psi \left( Z\right) -\mathbb{M}\psi\left(
Z\right) \right\vert >2\lambda R\right\} \leq 2C\exp \left( -c\lambda
^{2}\right).
\]
By symmetry, $\mathbb{M}\psi \left( Z\right) =0$. Now, remembering that $R$ is bounded in terms of $\lambda $, see (\ref{R lambda}), set $%
t=\lambda R$ and estimate $\lambda $ in terms of $t$ using Lemma \ref{lil
alg}. By increasing the value of $C$ in the probability bound as stated in the theorem, this probability bound becomes trivial unless the probability bound above is less than $1/2$ as required.
\end{proof}

\section{\label{mainy resty}Beyond convexity}

Here we elaborate on the discussion in Section \ref{non-convex}.

\subsection{The parameter $\frak{L}(\cdot)$}

For any set $A\subseteq \mathbb{R}^{n}$ containing at least two points, let
\[
\mathfrak{L}(A)=\sup\left\{\frac{\inf\left\{\textrm{Var}(f):f\in C([0,1],A),f(0)=x,f(1)=y \right\}}{\left\vert x-y \right\vert}:x,y\in A,x\neq y\right\}
\]
where $C([0,1],A)$ is the collection of all continuous functions $f:[0,1]\rightarrow A$ and $\textrm{Var}(f)$ is the total variation of $f$,
\[
\textrm{Var}(f)=\sup\left\{\sum_{i=1}^N\left\vert f(t_i)-f(t_{i-1})\right\vert:N\in\mathbb{N},0=t_0<t_1<... <t_N=1 \right\}.
\]
Note that the numerator in the definition of $\mathfrak{L}(A)$ is a sort of geodesic distance. We assume the usual convention that $\inf\emptyset=\infty$, so unless $A$ is path connected, $\mathfrak{L}(A)=\infty$. If $A$ contains fewer than two points, set $\mathfrak{L}(A)=1$.

Whenever $A$ is convex, $%
\mathfrak{L}(A)=1$. Conversely,

\begin{proposition}
 If $A\subseteq\mathbb{R}^n$ is closed and $\mathfrak{L}(A)=1$ then $A$ is convex. 
\end{proposition}

\begin{proof}
Consider any $x,y\in A$ with $x\neq y$, any $\lambda\in (0,1)$ and any $f\in C([0,1],A)$ such that $f(0)=x$, $f(1)=y$. By continuity there exists $t\in(0,1)$ and $\theta \perp (y-x)$ such that $f(t)=\lambda x+(1-\lambda)y+\theta$. Now $f(t)\in A\cap B(\lambda x+(1-\lambda)y,\left\vert\theta\right\vert)$ and
\begin{eqnarray*}
\textrm{Var}(f) &\geq& \sqrt{(1-\lambda)^2\left\vert x-y\right\vert^2+\left\vert\theta\right\vert^2}+\sqrt{\lambda^2\left\vert x-y\right\vert^2+\left\vert\theta\right\vert^2}
\end{eqnarray*}
and the result follows by taking $\textrm{Var}(f)\rightarrow \left\vert x-y\right\vert$ and using the fact that $A$ is closed.
\end{proof}

Let $Lip\left( \psi,x\right)$ denote the local Lipschitz constant of a
function $\psi:A\rightarrow \mathbb{R}$ around a point $x\in A$,%
\begin{equation}
Lip\left( \psi,x\right) =\lim_{\varepsilon \rightarrow 0^{+}}Lip\left(
\psi|_{B\left( x,\varepsilon \right) \cap A}\right)  \label{loc lip def}
\end{equation}%
where $B\left( x,\varepsilon \right)$ denotes the Euclidean ball centred at $x$ of radius $\varepsilon$. When $A$ contains a neighbourhood of $x\in\mathbb{R}^n$ and $\nabla\psi$ is continuous at $x$, then $Lip\left( \psi,x\right)=\left\vert\nabla\psi(x)\right\vert$. Our main reason for defining $\mathfrak{L}(\cdot )$ is the following
observation.

\begin{proposition}
\label{lip grad}For any non-empty set $A\subseteq \mathbb{R}^{n}$ with $\mathfrak{L}(A)<\infty $ and
any function $\psi:A\rightarrow \mathbb{R}$,%
\begin{equation}
Lip\left( \psi\right) \leq \mathfrak{L}(A)\sup \left\{ Lip\left( \psi,x\right)
:x\in A\right\}.\label{lip bound by local}
\end{equation}%
If, furthermore, $A$ contains at least two points and is locally convex in the sense that for all $x\in A$
there exists $\varepsilon >0$ such that $B\left( x,\varepsilon \right) \cap
A $ is convex, then 
\begin{equation*}
\mathfrak{L}(A)=\sup_{\psi:A\rightarrow \mathbb{R}}\left\{ \frac{Lip\left(
\psi\right) }{\sup \left\{ Lip\left( \psi,x\right) :x\in A\right\} }:0<Lip\left(
\psi\right) <\infty \right\}.
\end{equation*}
\end{proposition}

\begin{proof}
We start with (\ref{lip bound by local}). Let $L=\sup \left\{ Lip\left( \psi,x\right)
:x\in A\right\}$. Consider any $x,y\in A$, any $\varepsilon>0$, and any continuous path $f:[0,1]\rightarrow A$ between $x$ and $y$. It follows from the definition of $L$ and of the local Lipschitz constant that for each $z\in f([0,1])$ there exists a nonempty open ball $B_z$ centered at $z$ with $Lip\left(\psi|_{B_z}\right)<L+\varepsilon$. Let $\mathcal{I}$ be the collection of all connected components of inverse images of these balls (each will be an interval that is open in the subspace topology of $[0,1]$). The elements of $\mathcal{I}$ cover $[0,1]$ so by compactness there is a finite subcover $\mathcal{I}^*\subseteq\mathcal{I}$. After removing unnecessary elements of $\mathcal{I}^*$ one by one, we arrive at a minimal subcover $\mathcal{I}^{**}$. The elements of $\mathcal{I}^{**}$ can then be ordered as follows: $I_1< I_2$ if there exists $t_1\in I_1$ such that for all $t_2\in I_2$, $t_1<t_2$, and $I_1\leq I_2$ if either $I_1<I_2$ or $I_1=I_2$. By minimality this is a linear ordering. The elements of $\mathcal{I}^{**}$ can then be labelled $I_1, I_2... I_N$, and we write $I_1=[a_1,b_1)$, $I_i=(a_i,b_i)$ for $2\leq i\leq N-1$, and $I_N=(a_N,b_N]$, with $a_1=0$ and $b_N=1$. We assume that $N\geq 3$ and leave the simpler case $N\in\{1,2\}$ to the reader. Now $a_{i+1}<b_i$ for all $1\leq i\leq N-1$, otherwise $b_i\notin\cup I_j$, so there exists $c_i\in (a_{i+1},b_i)$. By minimality, $b_{i-1}<a_{i+1}$ for $2\leq i \leq N-1$ (otherwise $I_i\subseteq I_{i-1}\cup I_{i+1}$). So, setting $c_0=0$ and $c_N=1$, the sequence $(c_i)_0^{N}$ is strictly increasing. It also follows that for all $1\leq i \leq N$, $[c_{i-1},c_i]\subseteq I_i$, and from the construction of $\mathcal{I}$ and $\mathcal{I}^{**}$, that $\psi$ is $(L+\varepsilon)$-Lipschitz on each $f(I_i)$. This implies that $\left\vert \psi(f(c_i))-\psi(f(c_{i-1}))\right\vert \leq (L+\varepsilon)\left\vert f(c_i)-f(c_{i-1})\right\vert$, and by the triangle inequality $\left\vert \psi(f(0))-\psi(f(1))\right\vert\leq (L+\varepsilon)\sum_{i=1}^N \left\vert f(c_{i-1})-f(c_i)\right\vert\leq (L+\varepsilon)\textrm{Var}(f)$.  We now choose $f$ such that $\textrm{Var}(f)\leq (\mathfrak{L}(A)+\varepsilon)\left\vert x-y\right\vert$ and we let $\varepsilon\rightarrow0$. This implies that $\left\vert \psi(x)-\psi(y)\right\vert \leq \mathfrak{L}(A)L\left\vert x-y\right\vert$, which implies (\ref{lip bound by local}). This also implies a lower bound for $\mathfrak{L}(A)$.

We now prove the second part
under the assumption that $A$ is locally convex, as defined in the statement
of the theorem. For any $x,y\in A$ let
\[
\rho (x,y)=\inf\left\{\textrm{Var}(f):f\in C([0,1],A),f(0)=x,f(1)=y \right\}
\]
It is
easily seen that $\rho $ is a metric on $A$. Now consider any $\varepsilon
>0 $. It follows from the definition of $\mathfrak{L}%
(A) $, that there exist $x,y\in A$ such that $\rho (x,y)>\left( \mathfrak{L}%
(A)-\varepsilon \right) \left\vert x-y\right\vert $. Now consider the
function $g:A\rightarrow \lbrack 0,\infty )$ defined as $g(z)=\rho (x,z)$,
for which $g(x)=0$, $g(y)=\rho (x,y)$, so $Lip(g)>\mathfrak{L}%
(A)-\varepsilon $. It follows from the definition of $\rho $, the triangle
inequality, and the assumption of local convexity that $g$ is locally $1$%
-Lipschitz. This shows that%
\begin{equation*}
\mathfrak{L}(A)<\frac{Lip(g)}{\sup \left\{ Lip\left( g,x\right) :x\in
A\right\} }+\varepsilon
\end{equation*}%
and the result follows by sending $\varepsilon \rightarrow 0$.
\end{proof}

Consider the case where $A$ is a non-affine deformation of a
convex body $K$, as the inverse image under the action of some continuous map $T:%
\mathbb{R}^{n}\rightarrow \mathbb{R}^{n}$. When $K$ is $1$-unconditional
(i.e. invariant under coordinate reflections) and $T$ acts coordinatewise
and monotonically, then $\mathfrak{L}(A)\leq \sqrt{2}$. So, for example, $\mathfrak{L}(B_p^n)\leq \sqrt{2}$ for all $p\in(0,1)$, where
\[
B_p^n=\left\{x\in \mathbb{R}^n:\sum_{i=1}^n \left\vert x_i \right\vert^p \leq1\right\}.
\]
The general case is covered by the following proposition.

\begin{proposition}
\label{distorted boxy}Let $n\in \mathbb{N}$ and let $\Upsilon _{n}$ denote
the collection of all $K\subseteq \mathbb{R}^{n}$ with the following
property: there exists $a\in \left[ -\infty ,\infty \right] ^{n}$ (depending
on $K$) such that if $x\in K$ and $y\in \mathbb{R}^{n}$, and each coordinate
of $y$ is between the corresponding coordinates of $x$ and $a$ in the
non-strict sense (i.e. either $a_{i}\leq y_{i}\leq x_{i}$ or $x_{i}\leq
y_{i}\leq a_{i}$), then $y\in K$. For each $1\leq i\leq n$, let $h_{i}:%
\mathbb{R}\rightarrow \mathbb{R}$ be a continuous non-decreasing function, and let $T:%
\mathbb{R}^{n}\rightarrow \mathbb{R}^{n}$ be defined as $%
Tx=(h_{i}(x_{i}))_{1}^{n}$. Then for any $K\in \Upsilon _{n}$, $\mathfrak{L}%
\left( K\right) \leq \sqrt{2}$ and $T^{-1}K\in \Upsilon _{n}$, and so $%
\mathfrak{L}\left( T^{-1}(K)\right) \leq \sqrt{2}$.
\end{proposition}

\begin{proof}
Consider any $K\in \Upsilon _{n}$. We first show that $T^{-1}K\in \Upsilon
_{n}$. Since $K\in \Upsilon _{n}$ $\exists a\in \left[ -\infty ,\infty %
\right] ^{n}$ as in the statement of the theorem. For each $1\leq i\leq n$,
since $h_{i}$ is non-decreasing, there exists $b_{i}\in \left[ -\infty
,\infty \right] $ such that for all $t\in \mathbb{R}$, if $t\leq b_{i}$,
then $h_{i}(t)\leq a_{i}$, and if $t\geq b_{i}$ then $h_{i}(t)\geq a_{i}$
(consider three cases: $a_{i}$ is an upper bound for $range\left(
h_{i}\right) $, $a_{i}$ is a lower bound, or neither). Now consider any $%
x\in T^{-1}\left( K\right) $ and $y\in \mathbb{R}^{n}$, such that the
coordinates of $y$ are between the corresponding coordinates of $x$ and
those of $b$ (always meant in the non-strict sense). By the fact that the $%
h_{i}$ are non-decreasing, and by construction of $b$, it follows that the
coordinates of $Ty$ are all between the coordinates of $Tx\in K$ and those of $%
a $. Since $K\in \Upsilon _{n}$, what we have just shown implies that $Ty\in
K$, and therefore $y\in T^{-1}K$. This shows that $T^{-1}K$ satisfies the
defining property of $\Upsilon _{n}$. We now show that $\mathfrak{L}(K)\leq 
\sqrt{2}$. If $K$ is empty, or a singleton, then $\mathfrak{L}\left(
K\right) =1$, and we may assume without loss of generality that $\left\vert
K\right\vert \geq 2$. Consider any $x,y\in K$ with $x\neq y$. Now define $%
z\in \mathbb{R}^{n}$ as follows. If $a_{i}$ is between $x_{i}$ and $y_{i}$
(which is only possible if $a_{i}\in \mathbb{R}$), then set $z_{i}=a_{i}$
(and let the collection of all such $i$ be denoted $E$), otherwise let $%
z_{i} $ be the element of the set $\left\{ x_{i},y_{i}\right\} $ that is
closest to $a_{i}$, with the obvious interpretation when $a_{i}\in \left\{
\pm \infty \right\} $. For all $\lambda \in \left[ 0,1\right] $ and all $%
1\leq i\leq n$, $\lambda z_{i}+\left( 1-\lambda \right) x_{i}$ is between $%
a_{i}$ and $x_{i}$, and therefore $\lambda z+\left( 1-\lambda \right) x\in K$%
. Similarly, $\lambda y+\left( 1-\lambda \right) z\in K$, and this defines a
polygonal path of length $\left\vert x-z\right\vert +\left\vert
y-z\right\vert $ in $K$ from $x$ to $y$. Furthermore,%
\begin{equation*}
\left\langle x-z,y-z\right\rangle =\sum_{i\in E}\left( x_{i}-a_{i}\right)
\left( y_{i}-a_{i}\right) \leq 0.
\end{equation*}%
Using this inequality and comparing the $\ell _{2}^{2}$ and $\ell _{1}^{2}$
norms,%
\begin{equation*}
\left\vert x-y\right\vert ^{2}=\left\vert x-z\right\vert ^{2}+\left\vert
y-z\right\vert ^{2}-2\left\langle x-z,y-z\right\rangle \geq \frac{1}{2}%
\left( \left\vert x-z\right\vert +\left\vert y-z\right\vert \right) ^{2}
\end{equation*}%
and it follows that $\mathfrak{L}(K)\leq \sqrt{2}$.
\end{proof}

For a function $\psi:\mathbb{R}^{n}\rightarrow \mathbb{R}$
define%
\begin{equation}
\mathfrak{L}(\psi)=\sup_{t\in \mathbb{R}}\mathfrak{L}\left( \left\{\psi\leq t\right\} \right).  \label{LLL def}
\end{equation}%
The following result summarizes the methodology outlined in Section \ref{methaa}.

\begin{theorem}
\label{main zeroth}Let $n\in \mathbb{N}$, $A>0$, let $\mu $ be a probability
measure on $\mathbb{R}^{n}$, and let $T:\mathbb{R}^{n}\rightarrow \mathbb{R}%
^{n}$ and $f ,Q:\mathbb{R}^{n}\rightarrow \mathbb{R}$ be measurable
functions such that $\mu =T\gamma _{n}$, and such that $%
f \circ T$ is locally Lipschitz with%
\begin{equation*}
Q\left( x\right) \geq Lip\left(f \circ T,x\right)
\end{equation*}%
for all $x\in \mathbb{R}^{n}$. Let $X$ and $Z$ be random vectors in $\mathbb{%
R}^{n}$, where the distribution of $X$ is $\mu $, and $Z$ follows the
standard normal distribution. Let $R>0$ and $t>\Phi ^{-1}\left( 1-\left(
2A+4\right) ^{-1}\right) $ be such that $\mathbb{P}\left\{ Q(Z)>R\right\}
\leq A\left( 1-\Phi (t)\right) $. Then%
\begin{equation*}
\mathbb{P}\left\{ \left\vert f (X)-\mathbb{M}f (X)\right\vert >2%
\mathfrak{L}(Q)Rt\right\} \leq \left( A+2\right) \left( 1-\Phi (t)\right)
\end{equation*}%
where $\mathfrak{L}(Q)$ is defined by (\ref{LLL def}).
\end{theorem}

\begin{proof}[Proof of Theorem \protect\ref{main zeroth}]
Since the distribution of the random vector $T(Z)$ is $\mu $, we may assume
without loss of generality that $X=T(Z)$. Set $\psi=f \circ T$%
, in which case $\psi(Z)=f (X)$. Let $K=\left\{ x:Q(x)\leq
R\right\} $. By assumption, $\mathbb{P}\left\{ Z\in K\right\} \geq1-A\left(
1-\Phi (t)\right) $, and for all $x\in K$, $Lip\left( \psi|_{K},x\right) \leq Lip\left( \psi,x\right) \leq R$. By
Proposition \ref{lip grad}, $Lip\left( \psi|_{K}\right) \leq 
\mathcal{L}\left( Q\right) R$. As in the proof of Theorem \ref{Weibul conc theory}, the function $\psi|_{K}$ may
be extended to a function $\widetilde{\psi}:\mathbb{R}^{n}\rightarrow 
\mathbb{R}$ such that $Lip\left( \widetilde{\psi}\right) =Lip\left( \psi|_{K}\right) $. By Gaussian concentration of Lipschitz functions and the union
bound, it follows that with probability at least $1-\left( A+2\right) \left(
1-\Phi (t)\right) $, $\psi(Z)=\widetilde{\psi}(Z)$, and $%
\left\vert \widetilde{\psi}(Z)-\mathbb{M}\widetilde{\psi}(Z)\right\vert
\leq \mathcal{L}\left( Q\right) Rt$. Since $\left( A+2\right) \left( 1-\Phi
(t)\right) <1/2$, this implies that greater than 50\% of the mass of the
distribution of $\psi(Z)$ lies in the closed interval from $%
\mathbb{M}\widetilde{\psi}(Z)-\mathcal{L}\left( Q\right) Rt$ to $\mathbb{M}%
\widetilde{\psi}(Z)+\mathcal{L}\left( Q\right) Rt$,
so $\left\vert \mathbb{M}\psi(Z)-\mathbb{M}\widetilde{\psi}%
(Z)\right\vert \leq \mathcal{L}\left( Q\right) Rt$. The result now follows
by the triangle inequality.
\end{proof}

\subsection{Convex subsets with large Gaussian measure}

The relevance of the following result is that, in the notation there, $\lambda\sqrt{n}B_2^n$ is an approximating ellipsoid for the possibly non-convex Orlicz ball $\lambda B_F$, in the sense of the Gaussian measure of the complement. The result applies more generally with essentially the same proof; we state it as is for simplicity.
\begin{proposition}\label{Orlicz approx}
Let $F:[0,\infty)\rightarrow[0,\infty)$ be strictly increasing and concave, such that $F(0)=0$, $F(1)=1$ and $\lim_{x\rightarrow\infty}F(x)=\infty$. Let $n\in\mathbb{N}$ and define
\[
B_F=\left\{x\in\mathbb{R}^n:\sum_{i=1}^nF\left(\left\vert x_i \right\vert\right)\leq n\right\}.
\]
Then for all $\lambda>1$ we have $\lambda\sqrt{n}B_2^n\subseteq \lambda B_F$ and
\[
\gamma_n\left(\mathbb{R}^n\setminus C\lambda\sqrt{n}B_2^n\right)\leq \gamma_n\left(\mathbb{R}^n\setminus \lambda B_F\right)
\]
where $C>1$ is a universal constant.
\end{proposition}

\begin{proof}
By convexity there exists $m\in(0,1)$ such that $F(s)\leq m(s-1)+1$ for all $s\in[0,\infty)$. Therefore, for any $x\in\sqrt{n}B_2^n$,
\begin{eqnarray*}
\sum_{i=1}^nF\left(\left\vert x_i \right\vert\right)\leq \sum_{i=1}^n\left(m(\left\vert x_i \right\vert-1)+1\right)\leq m\sqrt{n}\left\vert x \right\vert+n(1-m)\leq n
\end{eqnarray*}
so $\lambda\sqrt{n}B_2^n\subseteq \lambda B_F$ as claimed. By polar integration and using the tangent line approximation to a concave function,
\begin{eqnarray*}
\gamma_n\left(\mathbb{R}^n\setminus C\lambda\sqrt{n}B_2^n\right)&=&(2\pi)^{-n/2}n\cdot\textrm{vol}\left(B_2^n\right)\int_{C\lambda\sqrt{n}}^\infty e^{-r^2/2+(n-1)\ln r}dr\\
&\leq&\left(\frac{n}{2\pi}\right)^{n/2}\textrm{vol}\left(B_2^n\right)C^{n-2}\lambda^{n-2}\exp\left(\frac{-C^2\lambda^2n}{2}\right).
\end{eqnarray*}
On the other hand, since $F$ is strictly increasing, $\lambda B_F\subseteq \left\{x\in\mathbb{R}^n:\exists i,  \left\vert x_i \right\vert \leq \lambda\right\}$ so using the fact that the tails of the standard normal density are convex and therefore lie above their tangent lines,
\begin{eqnarray*}
\gamma_n\left(\mathbb{R}^n\setminus \lambda B_F\right)&\geq&\left(2\int_\lambda^\infty\frac{1}{\sqrt{2\pi}}e^{-t^2/2}dt\right)^n\geq \left(\frac{1}{\sqrt{2\pi}\lambda}e^{-\lambda^2/2}\right)^n
\end{eqnarray*}
and the result follows by taking $C$ large enough.
\end{proof}

For general star bodies it may not be as easy to find approximating convex subsets as in Proposition \ref{Orlicz approx}, although as $\lambda\rightarrow \infty$ this is straightforward. Let $\sigma_{n-1}$ denote Haar measure on $S^{n-1}$, normalized so that $\sigma_{n-1}\left(S^{n-1}\right)=1$.

\begin{proposition}\label{asy decay}
Let $n\in\mathbb{N}$ with $n\geq 2$, let $\left\Vert \cdot \right\Vert:S^{n-1}\rightarrow (0,\infty)$ be any bounded Borel measurable function, let $b=\sup(\left\Vert \cdot \right\Vert)$, and let $K$ be the unit ball of the homogeneous extension of $\left\Vert \cdot \right\Vert$, i.e.
\[
K=\{0\}\cup\left\{x\in\mathbb{R}^n\setminus\{0\}:\left\vert x \right\vert\left\Vert\frac{x}{\left\vert x \right\vert}\right\Vert\leq 1\right\}.
\]
For simplicity, assume that the $L^\infty$ norm of $\left\Vert \cdot \right\Vert$ (i.e. its essential supremum) is also $b$. Then for all $\lambda \geq\sqrt{2(n-1)}b$ and all $\delta\in(0,1)$,
\[
\gamma_n\left(\mathbb{R}^n\setminus R B_2^n\right)\leq \gamma_n\left(\mathbb{R}^n\setminus\lambda K\right)
\]
where
\[
R=\frac{\lambda}{(1-\delta)b}+\frac{2(1-\delta)b}{\lambda}\ln\left[\frac{4}{\sigma_{n-1}\left(\left\{\left\Vert \theta \right\Vert\geq (1-\delta)b\right\}\right)}\right].
\]
\end{proposition}

\begin{proof}
By polar integration, and using $t^{n-1}\geq r^{n-1}$ for $t\geq r\geq0$ and the standard normal tail bound from the proof of Proposition \ref{Orlicz approx}, and since $\lambda>b$, $\gamma_n\left(\mathbb{R}^n\setminus \lambda K\right)$ is bounded below by
\begin{eqnarray}
&&\frac{n\cdot\textrm{vol}\left(B_2^n\right)}{2(2\pi)^{n/2}}\int_{S^{n-1}}  \left(\frac{\lambda}{\left\Vert \theta \right\Vert}\right)^{n-2}\exp\left(-\frac{1}{2}\left(\frac{\lambda}{\left\Vert \theta \right\Vert}\right)^2\right)d\sigma_{n-1}(\theta)\nonumber\\
&\geq&\sigma_{n-1}\left(\left\{\left\Vert \theta \right\Vert\geq (1-\delta)b\right\}\right)\frac{n\cdot\textrm{vol}\left(B_2^n\right)}{2(2\pi)^{n/2}}r^{n-2}e^{-r^2/2}\label{integral measure bundis}
\end{eqnarray}
where $r=(1-\delta)^{-1}b^{-1}\lambda$. The inequality on the second line holds since $\lambda>b\sqrt{n-2}$ and $t^{n-2}e^{-t^2/2}$ is decreasing for $t\geq\sqrt{n-2}$. Since $R>r$, by log-concavity,
\begin{equation}
R^{n-2}e^{-R^2/2}\leq r^{n-2}e^{-r^2/2}\exp\left[-(R-r)\left(r-\frac{n-2}{r}\right)\right]\label{log concaviae}
\end{equation}
and it follows from the definition of $R$ that
\begin{equation}
\exp\left[-(R-r)\left(r-\frac{n-2}{r}\right)\right]\leq\frac{1}{4}\sigma_{n-1}\left(\left\{\left\Vert \theta \right\Vert\geq (1-\delta)b\right\}\right).\label{from R}
\end{equation}
It follows from (\ref{log concaviae}) and (\ref{from R}) that the quantity in (\ref{integral measure bundis}) is bounded below by
\[
\frac{2n\cdot\textrm{vol}\left(B_2^n\right)}{(2\pi)^{n/2}}R^{n-2}e^{-R^2/2}.
\]
Since $R\geq\sqrt{2(n-1)}$, it follows by polar integration and the upper bound for a standard normal tail based on log-concavity that this is bounded below by $\gamma_n\left(\mathbb{R}^n\setminus R B_2^n\right)$.
\end{proof}

\section*{Acknowledgements}

Thanks to Gusti van Zyl, Ramon van Handel and the anonymous referees for various comments and suggestions. Part of this work was done while the author was a postdoctoral
fellow at the Weizmann Institute of Science.

\end{document}